\documentclass[12pt]{article}
\usepackage[utf8]{inputenc}
\usepackage[english]{babel}
\usepackage{amsmath,amssymb,amsthm} \renewcommand\le{\leqslant}
\renewcommand\ge{\geqslant}

\theoremstyle{plain}
\newtheorem{theorem}{Theorem}
\newtheorem{lemma}{Lemma}
\newtheorem*{corollary}{Corollary}

\theoremstyle{definition}
\newtheorem*{definition}{Definition}
\theoremstyle{remark}

\newtheorem*{example}{Example}

\newcommand\R{\mathbb R}
\newcommand\eps{\varepsilon}

\DeclareMathOperator\id{id}
\DeclareMathOperator\diam{diam}
\DeclareMathOperator\dis{dis}
\DeclareMathOperator\level{level}

\title{Chain development of metric compacts\thanks{This research was supported by
Russian Science Foundation (project 14-50-00005)}}
\author{Yu.V.~Malykhin\thanks{Steklov Mathematical Institute},
E.V.~Schepin\thanks{Steklov Mathematical Institute}}

\begin{document}

\maketitle

\paragraph{Notions and basic facts.}

Let $(X,d)$ be a metric space. We call a sequence of points
$x=x_0,x_1,x_2,\ldots,x_n=y$ an $\eps$-chain if $d(x_i,x_{i+1})\le\eps$ for all
$i$. Define \textit{chain distance} $c(x,y)$ as the infimum of $\eps$ such that
there exists an $\eps$-chain from $x$ to $y$.

Chain distance satisfies strong triangle inequality:
$c(x,z)\le\max(c(x,y),c(y,z))$; hence it is ultrametric if it does not
degenerate. Obviously, $c=d$ if $d$ is already ultrametric.

\begin{definition}
A function $f\colon X\to\R$ is called \textit{chain development} if $f$
preserves chain distance:
$$
c(x,y)=\tilde c(f(x),f(y))\quad\mbox{for }x,y\in X,
$$
where $c$ is the chain distance on $(X,d)$ and $\tilde c$ is the chain distance
on the set $f(X)$ with usual distance $\tilde d(s,t)=|s-t|$.
\end{definition}

Chain development was firstly introduced by E.V.~Schepin for finite sets as a
tool for fast hierarchical cluster analysis. 
Note that chain development always exists for finite spaces and can be
effectively constructed using minumim weight spanning tree of the corresponding
graph; see~\cite{SM} and \cite[Section 4]{LemNew} for more details.
An equivalent construction appeared in the paper~\cite{TV} by A.F.~Timan and
I.A.~Vestfid: they proved that points of any finite ultrametric space can be
enumerated in a sequence $x_1,\ldots,x_n$ such that
$c(x_i,x_j)=\max(c(x_i,x_j),c(x_j,x_j))$ for $i<j<k$.

The goal of this paper is to discuss some properties of chain development for
infinite spaces. So, there are compacts with no chain developments, e.g. the
square $C\times C$ of a Cantor set. Necessary and sufficient condition of
existence of chain developments is given below in Theorem 2.

By \textit{diameter of a chain development} $f\colon X\to\R$ we mean 
$\diam f(X)=\sup f(X)-\inf f(X)$. It is proven in~\cite{SM} that for finite
spaces $X$ the diameter of chain developments is determined uniquely. It turns
out that this is not true in general case.

\begin{theorem}
    Let $(X,d)$ be a compact metric space. Then the diameter of chain
    developments (if there are any) is determined uniquely if and only if $X$ is
    countable.
\end{theorem}

Throughout this paper by $(Z,d)$ we denote a zero-dimensional compact
metric space. We focus on such spaces because study of chain developments for
arbitrary compacts essentially reduces to the zero-dimensional case.\footnote{One can
identify points of $(X,c)$ with $c(x,y)=0$ to obtain zero-dimensional
ultrametric compact $(Z_X,c)$; a chain development of $(X,d)$ exists if and only
if there is a chain development of ($Z_X,c$).}
We have the following property:
\begin{itemize}
    \item[(i)] $(Z,c)$ is an ultrametric space, i.e. chain distance does not
        degenerate.
\end{itemize}
Indeed, take $x,y\in Z$. The set $\{x\}$ is a connected component, hence $x\in
U\not\ni y$ for some closed open set $U$, so
$$
c(x,y) \ge \min_{\substack{u\in U\\v\in X\setminus U}} d(u,v) > 0.
$$

The transition from metric $d$ to ultrametric $c$ (which can be seen as a
functor) preserves topology:
\begin{itemize}
    \item[(ii)] The identity map $\id\colon Z\to Z$ is a homeomorphism between
        $(Z,d)$ and $(Z,c)$.
\end{itemize}
Indeed, $\id$ is 1-Lipshitz ($c(x,y)\le d(x,y)$), hence it is a continuous
bijection from compact to Hausdorff space, hence a homeomorphism.

\begin{itemize}
    \item[(iii)] Any chain development $f\colon Z\to\R$ is continuous (with usual topology on
        $\R$). Hence, $f(Z)$ is compact and $f$ is a homeomorphism between $Z$ and $f(Z)$.
\end{itemize}
Let $x_n\to x^*$ in $Z$; prove that $t_n := f(x_n)\to t^* =: f(x^*)$.
Suppose that $t_n>t^*+\eps$ for some $\eps>0$.
If there are no points of
$f(Z)$ in $(t^*,t^*+\eps)$, then $\tilde{c}(t_n,t^*)\ge\eps$ (where $\tilde c$
is the chain distance on $f(Z)$).  And if there is some $t=f(x)\in
(t^*,t^*+\eps)$, then $\tilde{c}(t_n,t^*)\ge \tilde{c}(t,t^*)=c(x,x^*)>0$.
In both cases $\tilde{c}(t_n,t^*)\not\to 0$, which contradicts that
$\tilde{c}(t_n,t^*)=c(x_n,x^*)\le d(x_n,x^*)\to 0$. So, $f$ is continuous.

The chain distance on a compact $K\subset\R$ is determined by the
lengths of the intervals of the open set $U_K:=[\min K,\max K]\setminus K$.
\begin{itemize}
    \item[(iv)] Chain distance between points $s,t$ of $K$ is equal to the maximal
        length of the intervals of $U_K$, lying between $s$ and $t$.
\end{itemize}

\paragraph{Existence of chain development.}

There is a well-known correspondence between ultrametric spaces and labeled
trees; here we describe it for our purposes. Let $(X,d)$ be a compact metric
space; we will construct a labeled tree $T(X,d)$ with a vertex set $V$ and a
labeling function $r\colon V\to\R$.  We take an arbitary point $v_0$ as a root
of our tree and assign to it the $c$-diameter of $X$, i.e. $r(v_0)
=\max\limits_{x,y\in X}c(x,y)$.  The relation
$c(x,y)<r(v_0)$ is an equivalence relation; hence, $X$ breaks into finite number of
``clusters'' $Q_1,\ldots,Q_n$ of points with pairwise chain distance less than
$r(v_0)$. Next,
we connect the root with $n$ children, say $v_1,\ldots,v_n$, with $v_j$
corresponding to $Q_j$. The we repeat the construction for each of $Q_j$: we assign
$r(v_j)=\max_{x,y\in Q_j}c(x,y)$, and connect $v_j$ with children
corresponding to the clusters $Q_{j,k}\subset Q_j$ with $c(x,y)<r(v_j)$, $x,y\in
Q_{j,k}$. And so on. The process stops if $c$-diameter of a cluster becomes zero.

So, with each vertex $v$ of $T(X,d)$ we associate:
\begin{itemize}
    \item $n(v)$~--- the number of children of $v$;
    \item $C(v)$~--- the set of children of $v$;
    \item $Q(v)$~--- the cluster of points, corresponding to $v$; e.g.
        $Q(v_0)=X$;
    \item $r(v)$~--- the $c$-diameter of $Q(v)$.
\end{itemize}

\begin{definition}
    The width of the space $(X,d)$ is defined as
    $$
    w(X,d) := \sum_v r(v)(n(v)-1),
    $$
    where the sum is over all vertices of the tree $T(X,d)$.
\end{definition}

\begin{theorem}
    Let $(X,d)$ be a compact metric space. Then there exists a chain
    development $f\colon X\to\R$ if and only if $w(X,d)<\infty$.  Moreover,
    $w(X,d)$ is the minimal possible diameter of a chain development of $X$.
\end{theorem}

The construction of the tree uses only the chain distance, so $T(Z,d)=T(Z,c)$
and $w(Z,d)=w(Z,c)$.
On the other hand, the ultrametric structure is fully captured by the tree
$T(Z,d)$. Each point $x\in Z$ lies in some sequence of clusters; hence, it
corresponds to a path in the tree.
\begin{lemma}
    Let $x,y\in Z$. If $x\ne y$, then they lie in diffenent path of the tree,
    and $c(x,y)$ is equal to $r(v)$, where $v$ is the lowest common ancestor of
    $x,y$, i.e. the farthest from root vertex lying on both paths.
\end{lemma}
    
\begin{proof}
    Assume $x,y$ lie in the same path $\{v_0,v_1,\ldots\}$ of the tree.
    The compactness of $Z$ implies that diameters of the clusters $Q(v_j)$ tend
    to zero. Then $c(x,y)$ is less than any diameter of the corresponding
    clusters, hence, $c(x,y)=0$, and $x=y$.
    
    Let $v$ be the lowest common ancestor of $x$ and $y$. Then
    $c(x,y)\le r(v)$ by the definition of $r(v)$ and
    $c(x,y)= r(v)$ because $x,y$ lie in different sub-clusters of $Q(v)$.
\end{proof}

Let us prove Theorem 2.
\begin{proof}
Consider the case of zero-dimensional ultrametric compact space $(Z,c)$.
The construction of the set $f(Z)$ is equivalent to the
construction of the tree $T(Z,c)$. Pick an interval $[a,b]$ of length
$w(Z,c)$; we know that
$$
w(Z,c) = \sum_{v\in C(v_0)}w(Q(v),c) + (n(v_0)-1) r(v_0).
$$
One can remove $n(v_0)-1$ disjoint open intervals of length $r(v_0)$ from
$[a,b]$ so that the remaining $n(v_0)$ closed intervals will have lengths
$\{w(Q(v),c)\}_{v\in C(v_0)}$.  Those closed intervals correspond to each of
$Q(v)$ and we proceed with them as with $[a,b]$.

After removal all of the open intervals we arrive at some closed set
$K\subset[a,b]$.
Every point $x\in Z$ corresponds to a path in $T(Z,c)$ and to a nested sequence of
closed intervals with non-empty intersection $t\in K$; we put $f(x)=t$
(intersection is always a point because $\mu(K)=0$).
The proof that $f$ is chain development is straight-forward using Lemma 1 and
property (iv). Note that $\diam f(Z)=w(Z,c)$.

Now, let $f\colon Z\to\R$ be a chain development. 
Define
$$
U_{f(Z)}:=[\min f(Z),\max f(Z)]\setminus f(Z).
$$
We prove that
\begin{equation}\label{diam}
    w(Z,c) = \mu(U_{f(Z)}) = \diam f(Z) - \mu(f(Z)).
\end{equation}
Remind that $r(v_0)$ is the $c$-diameter of $Z$ and the $\tilde{c}$-diameter
of $f(Z)$. It is obvious from (iv) that there are exactly $n(v_0)-1$ 
intervals of $U$ of length $r(v_0)$. Repeating this argument with sets
$f(Q(v))$, $v\in C(v_0)$,
we will count all of the intervals of $U$ and found that each vertex $v$
corresponds to $n(v)-1$ intervals of $U$ of length $r(v)$.  That
implies~(\ref{diam}).  Hence, $w(Z,c)<\infty$ and $\diam f(Z)\ge w(Z,c)$.

The general case follows easily.
\end{proof}

Me will make use of the following standard construction.

\begin{lemma}
    Let $K$ be an uncountable compact in $[a,b]$. Then
    for any $c>0$ there is a continuous increasing function
    $\theta\colon[a,b]\to\R$ such that $\mu(\theta(K))=\mu(K)+c$ and
    $\mu(\theta(I))=\mu(I)$ for any interval $I\subset [a,b]\setminus K$.
\end{lemma}
\begin{proof}
    Write $K$ as $N\cup P$, where $N$ is countable and $P$ is perfect. Let
    $\varkappa\colon [a,b]\to[0,1]$ be an analog of the Cantor's ladder for the
    set $P$; we need that $\varkappa$ is continuous and non-decreasing, $\varkappa([a,b])=[0,1]$ and
    $\left.\varkappa\right|_I\equiv\mathrm{const}$ for any interval $I\subset
    [a,b]\setminus P$. It remains to take $\theta(t)=t+c\varkappa(t)$.
\end{proof}

Now we are ready to prove Theorem 1.

\begin{proof}
We consider only the zero-dimensional case.  If $Z$
is countable, then $\mu(f(Z))=0$ and from~(\ref{diam}) we get $\diam
f(Z)=w(Z,c)$.
Suppose $Z$ is uncountable. Take any chain development $f\colon Z\to\R$ and
apply Lemma 2 to $K=f(Z)$ with some $c>0$. Then $\theta\circ f$ gives us a chain
development with another diameter.
\end{proof}

It appears that the diameter of a chain development of an uncountable compact
may be any number greater or equal than $w(X,d)$.

\begin{example}
    Consider the set $C\times C$, where $C\subset[0,1]$ is the usual Cantor
    set. Let $d((x_1,y_1),(x_2,y_2)) = \max(|x_1-x_2|,|y_1-y_2|)$ for $(x_i,y_i)\in
    C\times C$.
    Then there is no chain development for the space $(C\times C,d)$.
\end{example}
\begin{proof}
Let us compute $w(C\times C,d)$. In the tree $T(C\times C,d)$ each node has
four children; for example, the children of the root correspond to the
clusters
\begin{equation}\label{comp}
\left(C\cap \left[\frac{2i}{3},\frac{2i+1}{3}\right]\right)\times
\left(C\cap \left[\frac{2j}{3},\frac{2j+1}{3}\right])\right),\quad i,j=0,1.
\end{equation}
We have $r(v_0)=1/3$ for the root $v_0$ and
$r(u)=\frac13 r(v)$ for each children $u$ of $v$, by self-similarity
of $C$. Hence, $w(C\times C,d)=\sum_{k=0}^\infty 4^k3^{-k} = \infty$
and the claim follows from Theorem 2.
\end{proof}

\paragraph{Measure of disconnectivity.}

\begin{definition}
    Let $(X,d)$ be a metric space. Define \textit{measure of disconnectivity} of
    $(X,d)$ as
    $$
    \dis(X,d) = \inf_{x_i\sim y_i}\sum_{i}d(x_i,y_i),
    $$
    where the infimum is taken over sequences (finite or infinite) or pairs
    $(x_i,y_i)\in X\times X$, such that the space $(X,d)$ with identified points
    $x_i\sim y_i$ is a connected topological space.
\end{definition}

This notion is closely related to the minimum spanning trees of graphs. Indeed,
if $X$ is finite, then $\dis(X,d)$ is equal to the weight of a minimum spanning
tree for $X$ (we regard points of $X$ as vertices and take weights of edges equal
to the correponding distances).

\begin{theorem}
    Let $(X,d)$ be a compact metric space. Then $\dis(X,d)=w(X,d)$.
\end{theorem}

We need one more notation for vertices of a tree $T(X,d)$: by 
$\level(v)$ we denote the length of the path from the root to $v$.

\begin{proof}
Note that for finite sets $X$ the theorem follows from~\cite{SM}. We prove there
that $w(X,d)$ is the diameter of any chain development of $X$, and it is clear from
the proof that it is equal to the weight of a minimum spanning tree of $X$.

Let us prove that $\dis(X,d)\ge w(X,d)$. Pick some $N\in\mathbb N$ and consider
all clusters $Q(v)$ with either $\level(v)=N$ or $\level(v)<N$ and $r(v)=0$. We
denote by $(X_N,c_N)$ the ultrametric space, which comes from $(X,c)$ when we
identify points in each cluster.  To make $X$ connected, we should connect all
of the mentioned clusters, so $\dis(X,d)\ge\dis(X_N,c_N)$.
For finite sets, $\dis=w$, so $\dis(X_N,c_N)=w(X_N,c_N)$.
Obviously,
$T(X_N,c_N)$ is obtained from $T(X,d)$ by deleting vertices of level $>N$, and
assigning $r(v)=0$ for the new leaves. So
$$
w(X_N,c)=\sum_{\level(v)<N} r(v)(n(v)-1)\to w(X,c)\quad\mbox{as
$N\to\infty$,}
$$
hence $\dis(X,d)\ge w(X,d)$.

Let us prove that $\dis(X,d)\le w(X,d)$. For each vertex $v$ we connect 
the clusters $\{Q(u)\}_{u\in C(v)}$ to each other by picking appropriate pairs $(x_i,y_i)\in
C(u')\times C(u'')$. It is easy to show that one can make the set of that
clusters connected using pairs with $\sum d(x_i,y_i)=r(v)(n(v)-1)$. In total, the
sum is $w(X,d)$.
Let us prove that the image $\widetilde X$ of $X$ after
projection $\pi\colon X\to\widetilde X$ of idendtification $x_i\sim y_i$, is connected. If $\widetilde
U\subset\widetilde X$ is non-emply, open and closed, then $U=\pi^{-1}\widetilde U\subset X$
is also non-empty, open and closed; besides that, if $x_i\sim y_i$ and $x_i\in U$, then $y_i\in
U$. It remains to prove that $U=X$.

If $x\in U$, then $x\in Q(v)\subset U$ for some $v$. Indeed, $\delta:=\min_{u\in
U,v\in X\setminus U}d(u,v)>0$, so if we take $Q(v)\ni x$ with sufficiently small
diameter, $r(v)<\delta$, then $Q(v)\subset U$. So, $U$ is a union of
clusters; since $U$ is compact, it is a finite union. Now one can prove via
induction on $N$ that for all $v$ of $\mathrm{level}\ge N$ either $Q(v)\subset
U$ or $Q(v)\cap U=\varnothing$. Indeed, $U$ is a union of finite number of
clusters, so this is true for large $N$.
Let us make an induction step from $N$ to $N-1$. Suppose there is $Q(v)$,
$\level(v)=N-1$, with $Q(v)\cap U\ne\varnothing$. We have $Q(v)=\sqcup_{u\in
C(u)}Q(u)$ so $Q(u')\cap U\ne\varnothing$ for some $u'\in C(v)$. As
$\level(u')=N$, $Q(u')\subset U$. There is some $u''\in C(v)$ and a pair $x_i\sim y_i$,
$(x_i,y_i)\in Q(u')\times Q(u'')$. As $x_i\in U$, we have $y_i\in U$ and
$Q(u'')\subset U$. As all the clusters $\{Q(u)\}_{u\in C(v)}$ are connected, we
will prove that $Q(u)\subset U$ for all $u\in C(v)$, i.e. $Q(v)\subset U$. The
claim follows.

Finally, $Q(v_0)\subset U$ so $U=X$ and $\tilde X$ is connected.
\end{proof}

\begin{corollary}
    For any metric compact $(X,d)$ three quantities are equal:
    \begin{itemize}
        \item the minimal diameter of a chain development of $X$;
        \item the width $w(X,d)$;
        \item the measure of disconnectivity $\dis(X,d)$.
    \end{itemize}
\end{corollary}
Note that first two quantities definitely have ultrametric nature, but this
is not obvious for the third quantity.

\end{document}